\documentclass[12pt]{amsart}
\renewcommand{\bar}{\overline}
\usepackage{amsmath,amsthm,amscd,euscript}
\renewcommand{\bar}{\overline}
\def \r{\mathbb R}

\def \q{\mathbb Q}
\def \z{\mathbb Z}

\def \d{\mathcal D}

\DeclareMathOperator{\sign}{sign}

\newtheorem{theorem}{Theorem}[section]
\newtheorem{lemma}[theorem]{Lemma}

\newtheorem{proposition}[theorem]{Proposition}
\newtheorem{corollary}[theorem]{Corollary}
\theoremstyle{remark}
\newtheorem{remark}[theorem]{Remark}
\theoremstyle{definition}
\newtheorem{definition}[theorem]{Definition}
\newtheorem{example}[theorem]{Example}
\newtheorem{problem}{Problem}
\newtheorem{conjecture}[problem]{Conjecture}

\title{On Asymptotic Reducibility in $SL(3,\z)$}
\author{Oleg~Karpenkov}
\date{18 May 2011}


\email[Oleg Karpenkov]{karpenk@mccme.ru}
\address{TU Graz /Kopernikusgasse 24, A 8010 Graz, Austria/}

\begin{document}
\input epsf

\begin{abstract}
Recently we showed that Hessenberg matrices are proper to
represent conjugacy classes in $SL(n,\z)$. In this paper we focus
on the reducibility properties in the set of Hessenberg matrices
of $SL(3,\z)$. We investigate the first interesting open case
here: the case of matrices having one real and two complex
conjugate eigenvalues.
\end{abstract}

\maketitle \tableofcontents

\section*{Introduction}\label{intro}

In current paper we study the $SL(3,\z)$ integer conjugacy classes
via {\it reduced representatives}. Recall that two matrices $M_1$
and $M_2$ in $SL(3,\z)$ are {\it integer conjugate} if there
exists a matrix $X$ in $GL(3,\z)$ such that
$$
M_2=XM_1X^{-1}.
$$
For matrices with a pair of complex conjugate eigenvalues we
discovered the following phenomenon: {\it Hessenberg matrices
distinguish corresponding conjugacy classes asymptotically}
(Theorem~\ref{maintheorem}). We show that a similar statement is
not true for the case of operators with three real eigenvalues.

\vspace{2mm}

{\noindent {\bf Background.} In classical approach to
$SL(n,\z)$-conjugacy problem one splits $SL(n,\q)$-conjugacy
classes into $SL(n,\z)$ conjugacy classes. After that the problem
is reduced to certain problems related to orders of algebraic
fields extended by the roots of characteristic polynomial of the
corresponding matrices (like computing their class numbers, etc.).

In~\cite{Karpenkov2012} we introduced multidimensional analog of
Gauss Reduction Theorem, it is an alternative approach to the
conjugacy problem (for two-dimensional Gauss Reduction Theory we
refer to~\cite{Karpenkov2010}, \cite{Lewis1997},
and~\cite{Manin2002}). In multidimensional Gauss Reduction Theory
the key role play Hessenberg matrices, they generalize reduced
matrices in Gauss Reduction Theory. Hessenberg matrices are
matrices that vanish below the superdiagonal (for more information
see in~\cite{Stoer2002}). They appear in the
work~\cite{Hessenberg1942} by K.~Hessenberg for the first time,
they were later used in QR-algorithms (\cite{Demmel1997},
\cite{Trefethen1997}, \cite{Ortega1963/1964}).
In~\cite{Karpenkov2012} we defined a natural notion of Hessenberg
complexity for Hessenberg matrices, which is a nonnegative integer
function, and showed that {\it each integer conjugacy class of
irreducible matrices has only finite number of Hessenberg matrices
with minimal Hessenberg complexity}. }

\vspace{2mm}

{\noindent {\bf Description of the paper.} We start in
Section~\ref{Definitions} with general definitions and notation.
Further in Section~\ref{results} we formulate main results of
current paper: Theorem~\ref{t2parab} on parabolic structure of the
sets of Hessenberg matrices with two complex conjugate
eigenvalues, and Theorem~\ref{maintheorem} on asymptotic
reducibility of matrices in these sets. In Section~\ref{parab} we
prove Theorem~\ref{t2parab}. Further in Section~\ref{kvcf} we show
some necessary tools that we use in the proof of
Theorem~\ref{maintheorem} (Markoff-Davenport characteristic,
Klein-Voronoi continued fractions, etc.). Then in
Section~\ref{structure} we give a proof of
Theorem~\ref{maintheorem}. Finally in Section~\ref{Open} we
formulate several open problems.

\vspace{2mm}

{\noindent {\bf Acknowledgment.} The work is partially supported
by FWF grant M~1273-N18. The author is grateful to
E.~I.~Pav\-lov\-skaya and  H.~W.~Lenstra for useful remarks.}

\section{Definitions and notation}\label{Definitions}

\subsection{$\varsigma$-reduced matrices}
A matrix $M$ of the form
$$
\left(
\begin{array}{ccc}
a_{11} &a_{12} &a_{13}\\
a_{21} &a_{22} &a_{23}\\
0       &a_{32} &a_{33}\\
\end{array}
\right)
$$
is called an {\it $($upper$)$ Hessenberg} matrix. We say that the
{\it Hessenberg type} of $M$ is
$$
\langle a_{11},a_{21}|a_{12},a_{22},a_{32} \rangle.
$$

\begin{definition}
A Hessenberg matrix in $SL(3,\z)$ is said to be {\it perfect} if
we have
$$
\begin{array}{l}
0\le a_{11}<a_{21};\\
0\le a_{12}<a_{32};\\
0\le a_{22}<a_{32}.\\
\end{array}
$$
\end{definition}

\begin{definition}
The {\it Hessenberg complexity} of a Hessenberg matrix $M$ of
Hessenberg type $\langle a_{11},a_{21}|a_{12},a_{22},a_{32}
\rangle$ is the number $a_{12}^2a_{23}$, we denote it by
$\varsigma(M)$.
\end{definition}
Notice that $\varsigma(M)$ equals the volume of the parallelepiped
spanned by the following vectors $v=(1,0,0)$, $M(v)$, and
$M^2(v)$.

\begin{definition}
We say that a perfect Hessenberg matrix $M$ is {\it
$\varsigma$-reduced} if its Hessenberg complexity is the least
possible. Otherwise we say that the matrix is {\it
$\varsigma$-nonreduced}.
\end{definition}

In~\cite{Karpenkov2012} we proved the following result.

\begin{theorem}
{\it i$)$}. Any conjugacy class of $SL(3,\z)$ contains a
$\varsigma$-reduced matrix.

{\it ii$)$}. The number of $\varsigma$-reduced matrices is finite
in any integer conjugacy class.
\end{theorem}

The results of current paper give evidences concerning the fact
that the majority of Hessenberg matrices with two complex
conjugate and one real eigenvalues are $\varsigma$-reduced.

\subsection{Perfect Hessenberg matrices of a given Hessenberg type}

In this paper we study three-dimensional perfect Hessenberg
$SL(3,\z)$-matrices with irreducible characteristic polynomials.
There are two main geometrically essentially different cases of
$SL(3,\z)$-matrices: the {\it real spectrum} (or {\it RS-} for
short) {\it case} when the characteristic polynomials of matrices
have only real eigenvalues, and the {\it nonreal spectrum} (or
{\it NRS-} for short) {\it case} of matrices with a pair of
complex conjugate and one real eigenvalues.

\vspace{2mm}

Denote the set of all $SL(3,\z)$-matrices of a fixed a Hessenberg
type $\Omega$ by $H(\Omega)$. Let $H_v(\Omega)$ be the subset of
all NRS-matrices in $NRS(\Omega)$.

\begin{definition}
Let $\Omega=\langle a_{11},a_{21}|a_{12},a_{22},a_{32}\rangle$.
Consider $v=(a_{13},a_{23},a_{33})$ such that the determinant of
the matrix $(a_{ij})$ equals 1. Denote
$$
H_\Omega^{v}(m,n)= \left(
\begin{array}{ccr}
a_{11} &a_{12} &a_{11}m+a_{12}n+a_{13}\\
a_{21} &a_{22} &a_{21}m+a_{22}n+a_{23}\\
0       &a_{32} &       a_{32}n+a_{33}\\
\end{array}
\right).
$$
\end{definition}

It is clear that
$$
H(\Omega)=\big\{
H_\Omega^{v}(m,n) \big| m\in\z, n\in\z
\big\}
$$
Here to choose $v$ means to choose the origin $O$ in the plane
$H(\Omega)$. So the set $H(\Omega)$ has the structure of
two-dimensional plane. We denote by $OMN$ the coordinate system
corresponding to the parameters $(m,n)$.

\vspace{2mm}

Let $\d^{v}_\Omega(m,n)$ denote the discriminant of the
characteristic polynomial of $H^{v}_\Omega(m,n)$. Then the set
$NRS(\Omega)$ is defined by the following inequality in variables
$n$ and $m$:
$$
\d^{v}_\Omega(m,n)<0.
$$

\begin{example} In Figure~\ref{family.1} we show the subset of
NRS-matrices $NRS(\langle 0,1|0,0,1\rangle)$. For this example we
choose $v=(0,0,1)$.
\end{example}

\begin{figure}
$$\epsfbox{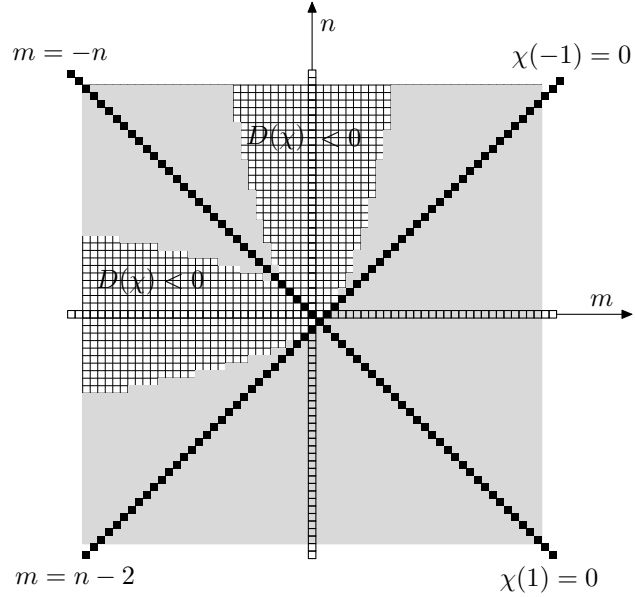}$$
\caption{The family of matrices of Hessenberg type $\langle
0,1|0,0,1\rangle$.}\label{family.1}
\end{figure}

\section{Formulation of main results and examples}\label{results}

We start in Subsection~\ref{parabolic_ssec} with the formulation
of a supplementary theorem on parabolic structure of the set of
NRS-matrices, we give the proof later in Section~\ref{parab}.
Further in Subsection~\ref{NonHypRay} we formulate the main result
on asymptotic uniqueness of $\varsigma$-reduced matrices, the
proof is shown in Section~\ref{structure}. In
Subsection~\ref{examplesHess} with describe examples of families
of matrices with fixed Hessenberg type.

\subsection{Parabolic structure of the set of NRS-matrices}\label{parabolic_ssec}
The set $NRS(\langle 0,1|0,0,1\rangle)$ on Figure~\ref{family.1}
''reminds'' the set of points with integer coordinates in the
union of the convex hulls of two parabolas. Let us formalize this
in a general statement.

\vspace{2mm}

Consider the matrix $H_{\Omega}^{v}(0,0)=(a_{ij})$ and define
$b_1$, $b_2$, and $b_3$ as coefficients of characteristic
polynomial of this matrix in variable $t$:
$$
-t^3+b_1t^2-b_2t+b_3.
$$
In the case of $SL(3,\z)$ we have $b_3=1$, nevertheless we write
$b_3$ for generality reasons. For the family $H^{v}_\Omega(m,n)$
we define the following two quadratic functions
$$
\begin{array}{l}
\displaystyle
p_{1,\Omega}(m,n)=m-\alpha_1n^2-\beta_1n-\gamma_1;\\
\displaystyle p_{2,\Omega}(m,n)=\frac{n}{a_{21}}-
\alpha_2\Big(\frac{a_{21}m-a_{11}n}{a_{21}}\Big)^2-
\beta_2\Big(\frac{a_{21}m-a_{11}n}{a_{21}}\Big)-\gamma_2,\\
\end{array}
$$
where
$$
\left\{
\begin{array}{l}
\displaystyle
\alpha_1=-\frac{a_{32}}{4a_{21}}\\
\displaystyle
\beta_1=\frac{a_{11}-a_{22}-a_{33}}{2a_{21}}\\
\displaystyle
\gamma_1=\frac{4b_2-b_1^2}{4a_{21}a_{32}}\\
\end{array}
\right.
;\qquad
\left\{
\begin{array}{l}
\displaystyle
\alpha_2=\frac{a_{32}a_{21}}{4b_3}\\
\displaystyle
\beta_2=-\frac{b_2}{2b_3}\\
\displaystyle
\gamma_2=\frac{b_2^2-4b_1b_3}{4a_{21}a_{32}b_3}\\
\end{array}
\right.
.
$$

Denote by $B_R(O)$ the interior of the circle of radius $R$
centered at the origin $(0,0)$ in the real plane $OMN$ of the
family $H^{v}_{\Omega}(m,n)$. For a real number $t$ we denote
$$
\Lambda_t=\{(m,n)\mid
(p_{1,\Omega}(m,n)-t)(p_{2,\Omega}(m,n)-t)<0\}.
$$

\begin{theorem}\label{t2parab}
For any positive $\varepsilon$ there exists $R>0$ such that in the
complement to $B_R(O)$ the following inclusions hold
$$
\Lambda_\varepsilon%
\subset NRS(\Omega) \subset
\Lambda_{-\varepsilon}.
$$
\end{theorem}
We give a proof of this theorem in Section~\ref{parab}.

\subsection{Theorem on asymptotic uniqueness of $\varsigma$-reduced
NRS-matrices}\label{NonHypRay}

A point is called {\it integer} if all its coordinates are
integers. A ray is said to be {\it integer} if its vertex is
integer and it contains integer points distinct to the vertex.

\begin{definition}
An integer ray in $H(\Omega)$ is said to be {\it an NRS-ray} if
all its integer points correspond to NRS-matrices. A direction is
said to be {\it asymptotic} for the set $NRS(\Omega)$ if there
exists an NRS-ray with this direction.
\end{definition}
As it is stated in Theorem~\ref{t2parab}, for any Hessenberg type
$\Omega$ the set $NRS(\Omega)$ almost coincides with the union of
the convex hulls of two parabolas. This implies the following
statement.

\begin{proposition}\label{asdir} Let $\Omega=\langle
a_{11},a_{21}|a_{12},a_{22},a_{32}\rangle$. There are exactly two
asymptotic directions for the set $NRS(\Omega)$, they are defined
by the vectors $(-1,0)$ and $(a_{11},a_{21})$. \qed
\end{proposition}

Let us consider a Hessenberg type $\Omega=\langle
a_{11},a_{21}|a_{12},a_{22},a_{32}\rangle$ and an appropriate
integer vector $v$.

\begin{definition}\label{rays}
Consider a family of Hessenberg matrices $H_\Omega^{v}$. Denote
$$
\begin{array}{l}
R^{m,n}_{1,\Omega,v}=\big\{H_{\Omega}^{v}(m{-}t,n)\big|t\in\z_{\ge 0}\};\\
R^{m,n}_{2,\Omega,v}=\big\{H_{\Omega}^{v}(m{+}a_{11}t,n{+}a_{21}t)\big|t\in\z_{\ge 0}\}.\\
\end{array}
$$
By $R^{m,n}_{1,\Omega,v}(t)$ or respectively by
$R^{m,n}_{2,\Omega,v}(t)$ we denote the $t$-th element in the
corresponding family.
\end{definition}

\begin{remark}
The families $R^{m,n}_{1,\Omega,v}$ and $R^{m,n}_{2,\Omega,v}$
coincide with the sets of all integer points of some rays with
directions $(-1,0)$ and $(a_{11},a_{21})$ respectively.
Conversely, from Proposition~\ref{asdir} it follows that the set
of integer points of any NRS-ray coincides either with
$R^{m,n}_{1,\Omega,v}$ or with $R^{m,n}_{2,\Omega,v}$ for some
integers $m$ and $n$.
\end{remark}

On Figure~\ref{NRSfig} we show in dark gray two NRS-rays:
$R^{-9,5}_{1,\langle 1,2|1,1,3 \rangle,(0,0,-1)}$ from the left
and $R^{-2,-1}_{2,\langle 1,2|1,1,3 \rangle,(0,0,-1)}$ from the
right.

\vspace{2mm}

Now we are ready to formulate the main result on asymptotic
behavior of NRS-matrices, we prove it later in
Section~\ref{structure}.

\begin{theorem}\label{maintheorem}{\bf(On asymptotic $\varsigma$-reducibility and uniqueness.)}
{\it i$)$}. Any NRS-ray $($as on Figure~\ref{NRSfig}$)$ contains
only finitely many $\varsigma$-nonreduced matrices.
\\
{\it ii$)$}. Any NRS-ray contains only finitely many matrices that
have more than one integer conjugate $\varsigma$-reduced matrix.
\end{theorem}

\begin{figure}
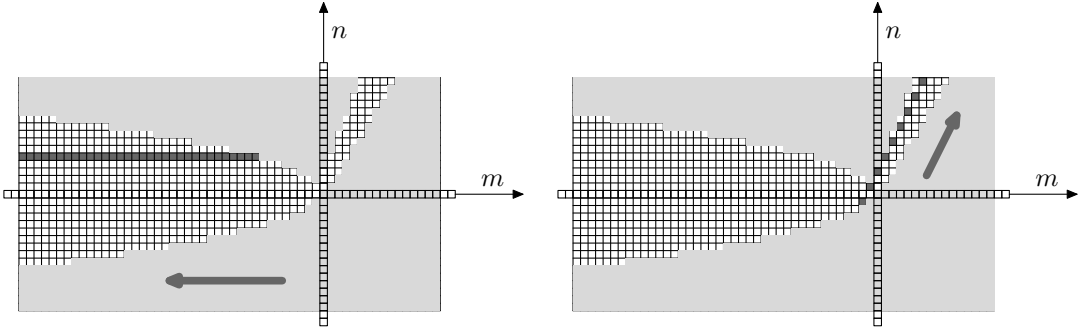

$$\epsfbox{result.1} \quad \epsfbox{result.2}$$
\caption{Any NRS-ray contains finitely many $\varsigma$-nonreduced
matrices.}\label{NRSfig}
\end{figure}

\begin{example}
Any NRS-ray for the Hessenberg type $\langle 0,1|0,0,1\rangle$
contains only $\varsigma$-reduced perfect matrices. Experiments
show that any NRS-ray for $\langle 0,1|1,0,2 \rangle$ contains at
most one $\varsigma$-nonreduced matrix (see in
Figure~\ref{family.2} on page~\pageref{family.2}).
\end{example}

\subsection{Examples of NRS-matrices for a given Hessenberg
type}\label{examplesHess} In this subsection we study several
examples of families $NRS(\Omega)$ for the Hessenberg types:
$$
\langle 0,1|0,0,1\rangle, \quad \langle 0,1|1,0,2\rangle, \quad
\langle 0,1|1,1,2\rangle, \quad \hbox{and} \quad \langle
1,2|1,1,3\rangle.
$$

In Figures~\ref{family.2}, \ref{family.3},
and~\ref{family.4} the dark gray squares correspond to
$\varsigma$-nonreduced matrices. We also fill with gray the
squares corresponding to $\varsigma$-reduced Hessenberg matrices
that are $n$-th powers ($n\ge 2$) of some integer matrices.

\vspace{2mm}

{\noindent{\bf Hessenberg perfect NRS-matrices $H_{\langle
0,1|0,0,1\rangle}^{(1,0,0)}(m,n)$.} The Hessenberg complexity of
all these matrices is 1, and, therefore, they are all
$\varsigma$-reduced, see the family on Figure~\ref{family.1} on
page~\pageref{family.1}.}

\vspace{2mm}

{\noindent{\bf Hessenberg perfect NRS-matrices $H_{\langle
0,1|1,0,2\rangle}^{(1,0,0)}(m,n)$.} The Hessenberg complexity of
these matrices equals $2$. Experiments show that 12 of such
matrices are $\varsigma$-nonreduced, see the family in
Figure~\ref{family.2}. It is conjectured that all others
Hessenberg matrices of $NRS\big(\langle 0,1|1,0,2\rangle\big)$ are
$\varsigma$-reduced.}

\vspace{2mm}

\begin{figure}
$$\epsfbox{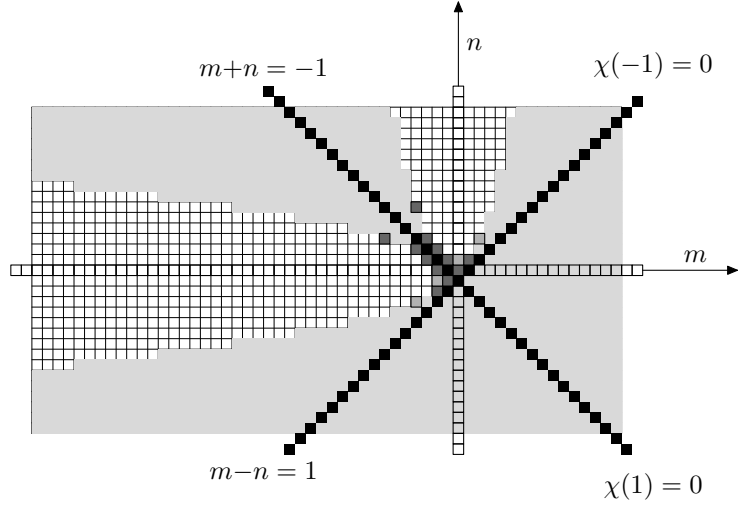}$$
\caption{The family of Hessenberg matrices $H_{\langle
0,1|1,0,2\rangle}^{(1,0,0)}(m,n)$.}\label{family.2}
\end{figure}

\vspace{2mm}

{\noindent{\bf Hessenberg perfect NRS-matrices $H_{\langle
0,1|1,1,2\rangle}^{(1,0,1)}(m,n)$.} The Hessenberg complexity of
these matrices equals $2$. We have found 12 $\varsigma$-nonreduced
matrices in the family. It is conjectured that all other
Hessenberg matrices of $NRS\big(\langle 0,1|1,1,2\rangle\big)$ are
$\varsigma$-reduced. See in Figure~\ref{family.3}.}

\begin{figure}
$$\epsfbox{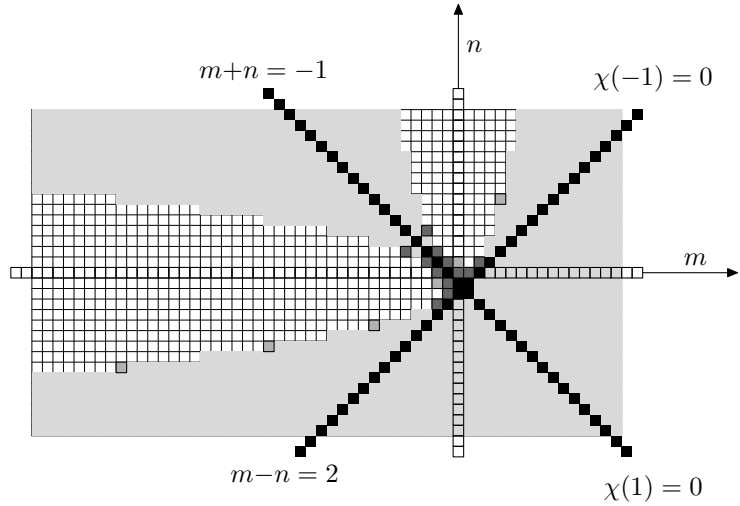}$$
\caption{The family of Hessenberg matrices $H_{\langle
0,1|1,1,2\rangle}^{(1,0,1)}(m,n)$.}\label{family.3}
\end{figure}

\vspace{2mm}

{\noindent{\bf Hessenberg perfect NRS-matrices $H_{\langle
1,2|1,1,3\rangle}^{(0,0,-1)}(m,n)$.} This is a more complicated
example of a family of Hessenberg perfect NRS-matrices, their
complexity equals $12$. We have found 27 $\varsigma$-nonreduced
matrices in the family. It is conjectured that all other
Hessenberg matrices of $NRS\big(\langle 1,2|1,1,3\rangle\big)$ are
$\varsigma$-reduced. See in Figure~\ref{family.4}.}

\begin{figure}
$$\epsfbox{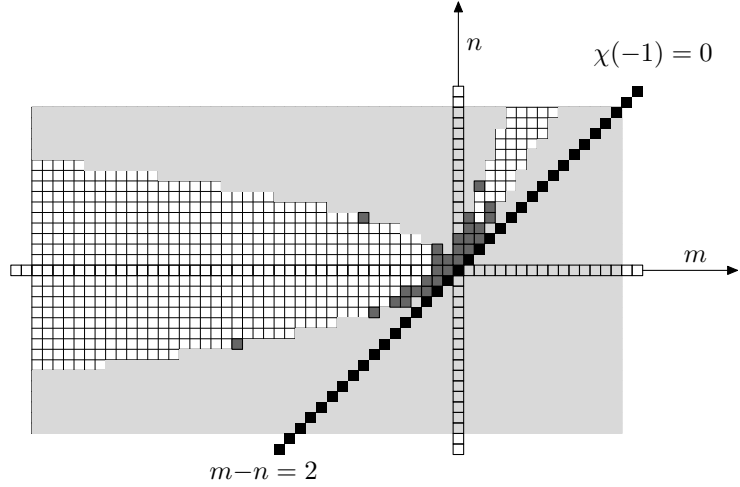}$$
\caption{The family of Hessenberg matrices $H_{\langle
1,2|1,1,3\rangle}^{(0,0,-1)}$.}\label{family.4}
\end{figure}

\section{Proof of Theorem~\ref{t2parab}}\label{parab}

We start the proof with several lemmas, but first let us give a
small remark.

\vspace{2mm}

{\noindent {\it Remark.} The set $NRS(\Omega)$ is defined by the
inequality
$$
\d_\Omega^{v}(m,n)<0.
$$
In the left part of the inequality there is a polynomial of degree
4 in variables $m$ and $n$. Note that the product
$16a_{21}^2a_{32}^2b_3\big(p_{1,\Omega}(m,n)p_{2,\Omega}(m,n)\big)$
is a good approximation to $\d_\Omega^{v}(m,n)$ at infinity: the
polynomial
$$
\d_\Omega^{v}(m,n)-16a_{21}^2a_{32}^2b_3\big(p_{1,\Omega}(m,n)p_{2,\Omega}(m,n)\big)
$$
is a polynomial of degree 2 in variables $m$ and $n$. }

\begin{lemma}\label{l2parab}
The curve $\d_{\langle 0,1|0,0,1\rangle}^{(1,0,0)}(m,n)=0$ is
contained in the domain defined by the inequalities:
$$
\left\{
\begin{array}{l}
(m^2-4n+3)(n^2+4m+3)\ge 0\\
(m^2-4n-3)(n^2+4m-3)-72\le 0\\
\end{array} \right.
$$
\end{lemma}

{\noindent {\it Remark.} Lemma~\ref{l2parab} implies that the
curve $\d_{\langle 0,1|0,0,1\rangle}^{(1,0,0)}(m,n)=0$ is
contained in some tubular neighborhood of the curve
$$
(m^2-4n)(n^2+4m)=0.
$$
}

\begin{proof}
Note that
$$
\d_{\langle
0,1|0,0,1\rangle}^{(1,0,0)}(m,n)=(m^2-4n)(n^2+4m)-2mn-27.
$$
Thus, we have
$$
\d_{\langle
0,1|0,0,1\rangle}^{(1,0,0)}(m,n)-(m^2-4n+3)(n^2+4m+3)=-2(n-3)^2-2(m+3)^2-(n+m)^2\le
0,
$$
and
$$
\d_{\langle
0,1|0,0,1\rangle}^{(1,0,0)}(m,n)-(m^2-4n-3)(n^2+4m-3)+72=2(n-3)^2+2(m+3)^2+(n-m)^2\ge
0.
$$
Therefore, the curve $\d_{\langle
0,1|0,0,1\rangle}^{(1,0,0)}(m,n)=0$ is contained in the domain
defined in the lemma.
\end{proof}

\begin{lemma}\label{l2parab2}
For any $\Omega=\langle a_{11},a_{21}|a_{12},a_{22},a_{32}
\rangle$ there exists an affine $($not necessarily integer$)$
transformation of the plane $OMN$ taking the curve
$\d_\Omega^{v}(m,n)=0$ to the curve $\d_{\langle
0,1|0,0,1\rangle}^{(1,0,0)}(m,n)=0$.
\end{lemma}

\begin{proof}
Let $H_\Omega^{v}(0,0)=(a_{i,j})$. Note that a matrix
$H_\Omega^{v}(m,n)$ is rational conjugate to the matrix
$$
H_{\langle 0,1|0,0,1\rangle}^{(1,0,0)}(a_{23}a_{32}-a_{11}a_{33}+
a_{12}a_{21}-a_{22}a_{33}-a_{11}a_{22}+a_{21}a_{32}m-a_{11}a_{32}n,
a_{11}+a_{22}+a_{33}+a_{32}n)
$$
by the matrix
$$
X_\Omega^{v}= \left(
\begin{array}{ccc}
1       &a_{11} &a_{11}^2+a_{12}a_{21} \\
0       &a_{21} &a_{11}a_{21}+a_{21}a_{22}\\
0       &0       &           a_{21}a_{32}\\
\end{array}
\right).
$$
Therefore, the curve $\d_\Omega^{v}(m,n)=0$ is mapped to the curve
$\d_{\langle 0,1|0,0,1\rangle}^{(1,0,0)}(m,n)=0$ bijectively. In
$OMN$ coordinates this map corresponds to the following affine
transformation
$$
\left(
\begin{array}{c}
m\\
n
\end{array}
\right) \mapsto
\left(
\begin{array}{c}
a_{21}a_{32}m-a_{11}a_{32}n\\
a_{32}n
\end{array}
\right) +
\left(
\begin{array}{c}
a_{23}a_{32}-a_{11}a_{33}+a_{12}a_{21}-a_{22}a_{33}-a_{11}a_{22}\\
a_{11}+a_{22}+a_{33}
\end{array}
\right).
$$

This completes the proof of the lemma.
\end{proof}

\vspace{1mm}

{\it Proof of Theorem~\ref{t2parab}.} Consider a family of
matrices $H_\Omega^{v}(-p_{1,\Omega}(0,t)+\varepsilon,t)$ with
real parameter $t$. Direct calculations show that for $\varepsilon
\ne 0$ the discriminant of the matrices for this family is a
polynomial of the forth degree in variable $t$, and
$$
\d_\Omega^{v}(-p_{1,\Omega}(0,t)+\varepsilon,t)=\frac{1}{4}a_{21}a_{32}^5\varepsilon
t^4+O(t^3).
$$
Therefore, there exists a neighborhood of infinity with respect to
the variable $t$ such that the function
$\d_\Omega^{v}(-p_{1,\Omega}(0,t)+\varepsilon,t)$ is positive for
positive $\varepsilon$ in this neighborhood, and negative for
negative $\varepsilon$.

Hence for a given $\varepsilon$ there exists a sufficiently large
$N_1=N_1(\varepsilon)$ such that for any $t>N_1$ there exists a
solution of the equation $ \d_\Omega^{v}(m,n)=0 $ at the segment
with endpoints
$$
\big(-p_{1,\Omega}(0,t)+\varepsilon,t\big) \quad \hbox{and} \quad
\big(-p_{1,\Omega}(0,t)-\varepsilon,t \big)
$$
of the plane $OMN$.

\vspace{2mm}

Now we examine the family in variable $t$ for the second parabola:
$$
\begin{array}{l}
H_\Omega^{v}\left(t{-}a_{11}p_{2,\Omega}(t,0){-}\frac{a_{11}}{\sqrt{a_{11}^2+a_{21}^2}}\varepsilon,
-a_{21}p_{2,\Omega}(t,0){-}\frac{a_{21}}{\sqrt{a_{11}^2+a_{21}^2}}\varepsilon\right).
\end{array}
$$

By the same reasons, for a given $\varepsilon$ there exists a
sufficiently large $N_2=N_2(\varepsilon)$ such that for any
$t>N_2$ there exists a solution of the equation $
\d^{v}_\Omega(m,n)=0 $ at the segment with endpoints
$$
\begin{array}{l}
\big(t{-}a_{11}p_{2,\Omega}(t,0){-}\frac{a_{11}}{\sqrt{a_{11}^2+a_{21}^2}}\varepsilon,
-a_{21}p_{2,\Omega}(t,0){-}\frac{a_{21}}{\sqrt{a_{11}^2+a_{21}^2}}\varepsilon\big)
\quad\hbox{ and }\\
\big(t{-}a_{,1}p_{2,\Omega}(t,0){+}\frac{a_{11}}{\sqrt{a_{11}^2+a_{21}^2}}
\varepsilon,-a_{21}p_{2,\Omega}(t,0){+}\frac{a_{21}}{\sqrt{a_{11}^2+a_{21}^2}}\varepsilon\big)
\end{array}
$$
of the plane $OMN$.

We have shown that for any of the four branches two parabolas
defined by $p_{1,\Omega}(m,n)=0$ and $p_{2,\Omega}(m,n)=0$ there
exists (at least) one branch of $\d^{v}_\Omega(m,n)=0$ contained
in the $\varepsilon$-tube of the chosen parabolic branch if we are
far enough from the origin.

\vspace{1mm}

From Lemma~\ref{l2parab} we know that $\d_{\langle
0,1|0,0,1\rangle}^{(1,0,0)}(m,n)=0$ is contained in some tubular
neighborhood of
$$
p_{1,\langle 0,1|0,0,1\rangle}(m,n)p_{2,\langle
0,1|0,0,1\rangle}(m,n)=0.
$$
Then by Lemma~\ref{l2parab2} the curve $\d_\Omega^{v}(m,n)=0$ is
contained in  some tubular neighborhood of the curve
$$
p_{1,\Omega}(m,n)p_{2,\Omega}(m,n)=0
$$
outside some ball centered at the origin. Finally, by Viet
Theorem, the intersection of the curve $\d_\Omega^{v}(m,n)=0$ with
each of the parallel lines
$$
\ell_t: \quad \frac{a_{11}+a_{21}}{a_{21}}n-m=t
$$
contains at most 4 points. Therefore, there exists sufficiently
large $T$ such that for any $t\ge T$ the intersection of the curve
$\d_\Omega^{v}(m,n)=0$ and $\ell_t$ contains exactly 4 points
corresponding to the branches of the parabolas
$p_{1,\Omega}(m,n)=0$ and $p_{2,\Omega}(m,n)=0$ lying in
$\Lambda_{-\varepsilon}\setminus\Lambda_{\varepsilon}$.

Hence, there exists $R=R(\varepsilon,N_1,N_2,T)$ such that in the
complement to the ball $B_R(O)$ we have
$$
\Lambda_\varepsilon
\subset NRS(\Omega) \subset
\Lambda_{-\varepsilon}.
$$
The proof of Theorem~\ref{t2parab} is completed. \qed

\section{Supplementary tools for the proof of Theorem~\ref{maintheorem}}\label{kvcf}

In this section we introduce several notions that we use in the
proof of Theorem~\ref{maintheorem}. In Subsection~\ref{md} we
introduce Markoff-Davenport characteristic that represents the
Hessenberg complexity. Further in Subsection~\ref{toHess} we show
how to construct perfect Hessenberg matrices $(M|v)$ conjugate to
a given one. Finally in Subsection~\ref{kv} we give the definition
of Klein-Voronoi continued fractions, formulate a theorem on
construction of $\varsigma$-reduced operators via vertices of a
fundamental domain of the corresponding Klein-Voronoi continued
fraction, and prove one supplementary statement on geometry of
continued fractions.

\subsection{MD-characteristics}\label{md} The study of the Markoff-Davenport
characteristics is closely related to the theory of minima of
absolute values of homogeneous forms with integer coefficients in
$n$-variables of degree~$n$. One of the first works in this area
was written by A.~Mar\-koff~\cite{Markoff1879} for the
decomposable forms (into the product of real linear forms) for
$n=2$. Further, H.~Davenport in series of
works~\cite{Davenport1938},~\cite{Davenport1938a},~\cite{Davenport1939},
\cite{Davenport1941}, and~\cite{Davenport1943} made first steps
for the case of decomposable forms for $n=3$.

\vspace{2mm}

Consider $A\in SL(n,\z)$. Denote by $P(A,v)$ the parallelepiped
spanned by vectors $v$, $A(v)$, $\ldots$, $A^{n-1}(v)$, i.e.,
$$
P(A,v)=\bigg\{O+\sum\limits_{i=0}^{n-1}\lambda_iA^{i}(v)\bigg|0\le
\lambda_i \le 1, i=0,\ldots, n{-}1\bigg\},
$$
where $O$ is the origin.
\begin{definition}
The {\it Markoff-Davenport characteristic} (or {\it
MD-characteristic}, for short) of an $SL(n,\z)$-operator $A$ is a
functional:
$$
\Delta_A:\r^n\to \r \qquad \hbox{defined by}
\qquad\Delta_A(v)=V(P(A,v)),
$$
where $V(P(A,v))$ is the nonoriented volume of $P(A,v)$.
\end{definition}

\begin{remark}\label{DavHes0}
Consider an operator $A$ with Hessenberg matrix $M$ in some
integer basis. Then the Hessenberg complexity $\varsigma(M)$
equals the value of MD-characteristic $\Delta_A(1,0,0)$.
\end{remark}

We continue with the following general definition.

\begin{definition}
The group of all $GL(3,\z)$-operators commuting with $A$ is called
the {\it Dirichlet group} and denoted by $\Xi(A)$.
\end{definition}


For MD-characteristic we have the following invariance property.

\begin{proposition}
Consider $A\in SL(n,\z)$ and let $B\in \Xi(A)$. Then for an
arbitrary $v$ we have
$$
\Delta_A(v)=\Delta_A(B(v)).
$$
\end{proposition}

Basically, this means that the MD-characteristic naturally defines
a function over the set of all orbits of the Dirichlet group.

\subsection{Construction of a perfect Hessenberg matrix $(M|v)$ conjugate to a given
one}\label{toHess}

Let us show how to construct perfect Hessenberg matrices integer
conjugate to a given one.

\vspace{2mm}

{\bf Algorithm to construct perfect Hessenberg matrices}.

 {\it Input Data.} We are given by an $SL(3,\z)$-matrix $M$ of an
operator $A$ with irreducible characteristic polynomial over $\q$
and an integer primitive $($i.e., with relatively prime
coordinates$)$ vector $v$.

{\it Step 1.} We put $e_1=v$.

{\it Step 2.} Choose an integer primitive vector of the plane
spanned by $v$ and $A(v)$ on the minimal possible nonzero
Euclidean distance from the line spanned by $v$, denote it by
$g_2$. Find the coordinates $q_{11}$ and $a_{21}$ from the vector
decomposition
$$
A(e_1)=q_{11} e_1+a_{21}g_{2}.
$$
Find $b_{11}$ and $a_{11}$ as integer quotients and reminders:
$$
q_{11}=|a_{21}|b_{11}+a_{11}.
$$
Define
$$
e_{2}=\sign(a_{21})g_{2}+b_{11} e_1.
$$

{\it Step 3.} Choose an integer primitive vector $g_3\in\r^3$ on
minimal possible nonzero Euclidean distance from the plane spanned
by $e_1$ and $e_2$. Find the coordinates $q_{12}$, $q_{22}$, and
$a_{32}$ from the vector decomposition
$$
A(e_2)=q_{12} e_1+q_{22} e_2+a_{32}g_{3}.
$$
Find $b_{12}$, $b_{22}$, $a_{12}$, and $a_{22}$ as integer
quotients and reminders:
$$
q_{12}=|a_{32}|b_{12}+a_{12} \quad \hbox{and} \quad
q_{22}=|a_{32}|b_{22}+a_{12}.
$$
Then we have
$$
e_{3}=b_{12} e_1+b_{22} e_2+\sign(a_{32})g_{3}.
$$

{\it Output Data.} Let $C$ be a transition matrix to the basis
$\{e_1,e_2,e_3\}$. In the output we have the perfect Hessenberg
matrix $CMC^{-1}$.

\begin{definition}
Consider an $SL(3,\z)$-matrix $M$ with irreducible characteristic
polynomial over $\q$ and an integer primitive vector $v$. Starting
from $M$ and $v$ the above algorithm generates a perfect
Hessenberg matrix, we denote it by $(M|v)$
\end{definition}

\begin{remark}
In~\cite{Karpenkov2012} we showed that any perfect Hessenberg
matrix integer conjugate to $M$ is represented as $(M|v)$ for a
certain integer primitive vector $v$.
\end{remark}

%
%
%
%

\subsection{Klein-Voronoi continued fractions}\label{kv}

In the proof of Theorem~\ref{maintheorem} we essentially use the
geometric construction of Klein-Voronoi continued fractions.
In~\cite{Klein1895} and~\cite{Klein1896} F.~Klein proposed a
multidimensional generalization of continued fractions to totally
real case. First attempts to find analogous construction in other
cases were made by G.~Voronoi in his
dissertation~\cite{Voronoui1952}. In 1985 J.~A.~Buchmann in his
papers~\cite{Buchmann1985} and~\cite{Buchmann1985a} proposed to
use Voronoi's extension to compute of fundamental units in orders.
We use a slightly modified definition of Klein-Voronoi continued
fraction from the paper~\cite{Karpenkov2012}.

\subsubsection{RS-case}

Let us first briefly recall Klein's definition of two-dimensional
continued fraction in totally real case. Consider an operator $A$
in $GL(3,\z)$ with three real distinct eigenvalues. This operator
has three distinct invariant planes passing through the origin.
The complement to the union of these planes consists of $8$ open
orthants. Let us choose an arbitrary orthant.

\begin{definition}
The boundary of the convex hull of all integer points except the
origin in the closure of the orthant is called the {\it sail}. The
set of all $8$ sails of the space $\r^{3}$ is called the {\it
$2$-dimensional continued fraction in the sense of Klein}.
\end{definition}

For further information on Klein continued fractions we refer to
the following literature:~\cite{Lachaud1993}, \cite{Arnold2002},
\cite{Arnold1998}, \cite{Arnold1999}, \cite{Kontsevich1999}
\cite{Korkina1996}, \cite{Korkina1995}, \cite{Lachaud2002},
\cite{Karpenkov2004}, \cite{Karpenkov2004a},
\cite{Karpenkov2009b}, \cite{Moussafir2000a} etc.

\subsubsection{NRS-case}\label{cf_1}

Consider an operator $A$ in $GL(3,\r)$ with distinct eigenvalues.
Suppose that it has a real eigenvalue $r$ and complex conjugate
eigenvalues $c$ and $\bar c$.

Denote by $T^1(A)$ the set of all real operators commuting with
$A$ such that they have a real eigenvalue equals $1$ are with
absolute value of both complex eigenvalues equal one. Actually,
$T^l(A)$ is an abelian group with operation of matrix
multiplication isomorphic to $S^1$.

For $v\in\r^3$ we denote
$$
T_A(v)=\{B(v)\mid B\in T^1(A)\}.
$$
If $v$ is a real eigenvector then $T_A(v)$ consists of one point.
Otherwise (in general case) $T_A(v)$ is homeomorphic to $S^1$.

Let $g_1$ be a real eigenvector with eigenvalue $r$; $g_2$ and
$g_3$ be vectors corresponding to the real and imaginary parts of
some complex eigenvector with eigenvalue $c$. Consider the
coordinate system $OXYZ$ corresponding to the basis $\{g_i\}$.
Denote by $\pi$ the $(k{+}l)$-dimensional plane $OXY$, and by
$\pi_+$ --- the half-plane of $\pi$ defined by $y\ge 0$.

\begin{proposition}
For any $v$ the orbit $T_A(v)$ intersects the half-plane $\pi_+$
in a unique point. \qed
\end{proposition}

\begin{definition}
A point $p\in \pi_+$ is said to be {\it $\pi$-integer} if the
orbit $T_A(p)$ contains at least one integer point.
\end{definition}

The invariant hyperplane $x=0$ of operator $A$ divides $\pi_+$
into two arcwise connected components.

\begin{definition}
The convex hull of all $\pi$-integer points except the origin
contained in a given arcwise connected component is called a {\it
factor-sail} of the operator $A$. The set of both factor-sails is
said to be the {\it factor-continued fraction} for the operator
$A$.
\\
The union of all orbits $T_A(*)$ in $\r^n$ represented by the
points in the factor-sail is called the {\it sail} of the operator
$A$. The set of all sails is said to be the  {\it continued
fraction} for the operator $A$ in the sense of Klein-Voronoi (see
in Figure~\ref{complex.1} below).
\end{definition}

It is clear that the factor-sail is a broken line. The
corresponding sail is the surface of elliptic rotation of the
factor-sail around the eigenline of $A$. The cones corresponding
to rotation of edges (vertices) are called {\it factor-edges}
({\it factor-vertices}).

\subsubsection{Algebraic continued fractions}\label{cf_3}

Consider an operator $A$ in $GL(3,\z)$ with a real eigenvalue $r$
and two complex conjugate distinct eigenvalues $c$ and $\bar c$.
Suppose also that the characteristic polynomial of $A$ is
irreducible over $\q$.

The Dirichlet group $\Xi(A)$ (of $GL(3,\z)$-operators commuting
with $A$) takes the Klein-Voronoi continued fraction to itself but
maybe exchange the sails. By Dirichlet unit theorem (see
in~\cite{Borevich1966}) the Dirichlet group $\Xi (A)$ is always
homomorphic to $\z\oplus\z/2\z$.

\begin{definition}\label{defFD}
A {\it fundamental domain of the Klein-Voronoi continued fraction}
for $A$ is a collection of open orbit-vertices and orbit-edges
such that for any orbit-face $F$ of the continued fraction there
exists a unique orbit-face $F'$ in this collection and an operator
$T\in \Xi(A)$ such that $F=T(F')$.
\end{definition}

\begin{example}
Consider an operator
$$
A= \left(
\begin{array}{ccc}
0 &0 &1\\
1 &0 &1\\
0 &1 &3\\
\end{array}
\right).
$$
It has one real and two complex conjugate eigenvalues. In
Figure~\ref{complex.1}a we show in light gray the halfplane
$\pi_+$, the invariant plane for $A$ corresponding to complex
conjugate eigenvalues is colored in dark gray. The boundary line
of the halfplane $\pi_+$ is an invariant line of $A$, it contains
real eigenvectors of $A$.

The halfplane $\pi_+$ is shown in Figure~\ref{complex.1}b. The
invariant plane intersects $\pi_+$ in a ray separating $\pi_+$
into two connected components. A point of $\pi_+$ is colored in
black if and only if it is a $\pi$-integer point. The boundaries
of the convex hulls in each part of $\pi_+$ are two factor-sails.
Notice that, one factor-sail is taken to another by the induced
action of the operator $-Id$, where $Id$ is the identity operator.

In Figure~\ref{complex.1}c we show one of the sails of
Klein-Voronoi continued fraction for $A$. There are three visible
orbit-vertices, they correspond to integer vectors $(1,0,0)$,
$(0,1,0)$, and $(0,0,1)$: the large dark points $(0,1,0)$ and
$(0,0,1)$ are visible on the corresponding orbit-vertices, the
point $(1,0,0)$ is on the backside of the continued fraction.

A fundamental domain of the operator consists of one orbit-vertex
and one orbit edge. For instance, one can take the orbit-vertex
corresponding to the point $(1,0,0)$ and the orbit-edge
corresponding to the ''tube'' connecting orbit-vertices for the
points $(1,0,0)$ and $(0,1,0)$.

\begin{figure}
$$\epsfbox{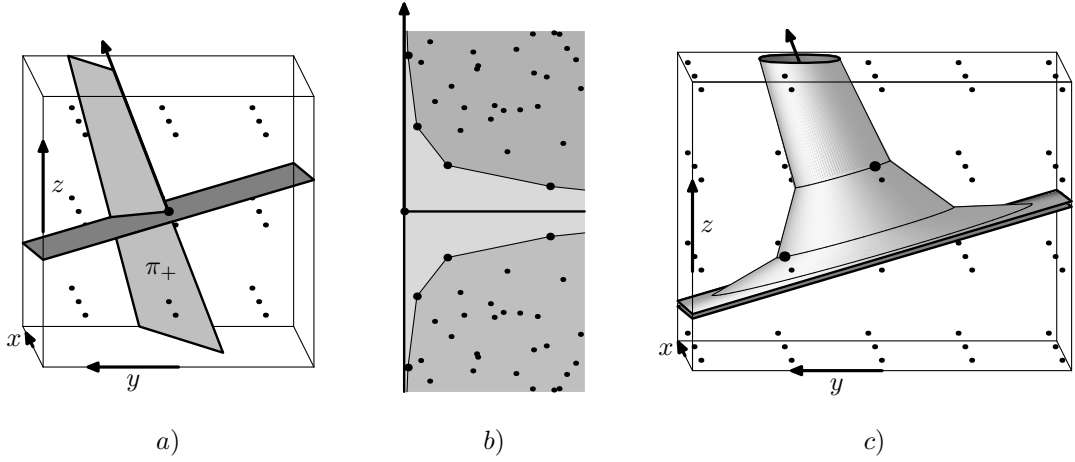}$$
\caption{A Klein-Voronoi continued fraction: a) the cone $\pi_+$
and the eigenplane; b) the continued factor-fraction; c) one of
the sails.}\label{complex.1}
\end{figure}
\end{example}

In the proof of Theorem~\ref{maintheorem} we use the following
result on construction of $\varsigma$-reduced operators via
vertices of fundamental domains of Klein-Voronoi continued
fractions.

\begin{theorem}\label{cfr}{\bf(\cite{Karpenkov2012})}
Consider an $SL(n,\z)$=operator $A$ with matrix $M$ having
distinct eigenvalues. Let $U$ be a fundamental domain of the
Klein-Voronoi continued fractions for $A$. Then we have:

{\it $($i$)$} For any $\varsigma$-reduced matrix $\hat M$ integer
conjugate to $M$ there exists $v\in U$ such that $ \hat M=(M|v)$.

{\it $($ii$)$} Let $v\in U$. The matrix $(M|v)$ is
$\varsigma$-reduced if and only if the MD-characteristic
$\Delta_A(v)$ attains its minimal value. \qed
\end{theorem}

\subsubsection{One general fact on fundamental domain of
Klein-Voronoi continued fractions for NRS-matrices of $SL(3,\z)$}
Further we use the following statement.

Consider an NRS-operator $A$ in $SL(3,\z)$ and any integer point
$x$ distinct from the origin. Denote by $\Gamma_A^0(p)$ the convex
hull of the union of two orbits corresponding to the points $p$
and $A(p)$. For any integer $k$ we denote by $\Gamma_A^k(p)$ the
set $A^k(\Gamma_A^0(x))$.

\begin{proposition}\label{orbitsCF}
Let $A$ be an NRS-operator in $SL(3,\z)$ and $p$ be an integer
point distinct from the origin. Then there exists a fundamental
domain of the Klein-Voronoi continued fraction for $A$ with all
$($integer$)$ orbit-vertices contained in the set $\Gamma_A^0(p)$.
\end{proposition}

The proof is based on the following lemma. Let
$$
\Gamma_A(p)=\bigcup\limits_{k\in\z}\Gamma_A^k(p).
$$
\begin{lemma}\label{cfLemma}
Consider $A\in SL(3,\z)$ with NRS-matrix and let $p$ be any
integer point distinct from the origin. Then one of the
Klein-Voronoi sails for $A$ is contained in the set $\Gamma_A(p)$.
\end{lemma}

\begin{proof}
Notice that the set $\Gamma_A(p)$ is a union of orbits. Let us
project $\Gamma_A(p)$ to the halfplane $\pi_+$. The set
$\Gamma_A(p)$ projects to the closure of the complement of the
convex hull for the points $\pi(A^k(p))$ for all integer number
$k$ in the angle defined by eigenspaces. Since all the points
$A^k(p)$ are integer, their convex hull is contained in the convex
hull of all points corresponding to integer orbits in the angle.
Hence $\pi(\Gamma_A(p))$ contains the projection of the sail.
Therefore, the set $\Gamma_A(p)$ contains one of the sails.
\end{proof}

{\noindent {\it Proof of Proposition~\ref{orbitsCF}.} Since $-Id$
exchange the sails, one can choose a fundamental domain entirely
contained in one sail. Let this sail contains a point $p$. By
Lemma~\ref{cfLemma} $\Gamma_A(p)$ contains this sail. Therefore
all the orbit-vertices of a fundamental domain for Klein-Voronoi
continued fraction can be chosen from the factor-set
$\Gamma_A^0(p)$. \qed}

\section{Proof of Theorem~\ref{maintheorem}}\label{structure}

\subsection{Geometry of Klein-Voronoi continued fractions for
matrices of $R^{m,n}_{1,\Omega,v}$} Let us show the following
statement.

\begin{proposition}\label{cfProp}
Consider an NRS-ray $R^{m,n}_{1,\Omega,v}$. Then there exists
$C>0$ such that for any $t>C$ there exists a fundamental domain
for the Klein-Voronoi continued fraction of the matrix
$R^{m,n}_{1,\Omega,v}(t)$ such that all integer points in this
domain are contained in the triangle with vertices $(1,0,0)$,
$(a_{11},a_{21},0)$, and $(-a_{11},-a_{21},0)$.
\end{proposition}

We begin with the case of matrices of Hessenberg type
$\Omega_0=\langle 0,1|0,0,1\rangle$. Such matrices form a family
$H(\Omega_0)$ with real parameters $m$ and $n$ as before:
$$
H_{\langle 0,1|0,0,1\rangle}^{(1,0,0)}(m,n)= \left(
\begin{array}{ccc}
0 &0 &1\\
1 &0 &m\\
0 &1 &n\\
\end{array}
\right).
$$
Here $v_0=(1,0,0)$.

\begin{lemma}\label{triangle}
Let $R^{m,n}_{1,\Omega_0,v_0}$ be an NRS-ray. Then for any
$\varepsilon>0$ there exists $C>0$ such that for any $t>C$  the
convex hull of the union of two orbit-vertices
$$T_{R^{m,n}_{1,\Omega_0,v_0}(t)}(1,0,0) \quad \hbox{and}
\quad T_{R^{m,n}_{1,\Omega_0,v_0}(t)}(0,1,0)
$$
is contained in the $\varepsilon$-tubular neighborhood of the
convex hull of three points $(1,0,0)$, $(0,1,0)$, $(0,-1,0)$.
\end{lemma}

{\noindent {\it Remark.} Actually Lemma~\ref{triangle} means that
the corresponding domain tends to be flat while the parameter $t$
tends to infinity.}

\begin{proof}
Let us find the asymptotics of eigenvectors and eigenplanes for
operators $R^{m,n}_{1,\Omega_0,v_0}(t)$ while $t$ tends to
$+\infty$. Denote the real eigenvector of
$R^{m,n}_{1,\Omega_0,v_0}(t)$ by $e(t)$. We have
$$
e(t)=\mu\big((1,0,0)+O(t^{-1})\big)
$$
for some nonzero real $\mu$.

Consider the unique invariant real plane of the operator
$R^{m,n}_{1,\Omega_0,v_0}(t)$ (it corresponds to the pair of
complex conjugate eigenvalues). Notice that this plane is a union
of all orbits $T_{R^{m,n}_{1,\Omega_0,v_0}(t)}(w)$ for arbitrary
vectors $w$ of this plane. Any such orbit is an ellipse with axes
$\lambda g_{\max}(t)$ and $\lambda g_{\min}(t)$ for some positive
real number $\lambda$, where
$$
\begin{array}{l}
g_{\max}(t)=(0,t,0)+O(1),\\
g_{\min}(t)=(0,0,t^{1/2})+O(t^{-1/2}).\\
\end{array}
$$
Actually, the vectors $g_{\max}(t){\pm}I g_{\min}(t)$ are two
complex eigenvectors of $R^{m,n}_{1,\Omega_0,v_0}(t)$. For the
ratio of the lengths of maximal and minimal axes of any orbit we
have the following asymptotic estimate:
$$
\frac{\lambda|g_{\max}(t)|}{\lambda|g_{\min}(t)|}=|t|^{1/2}+O(|t|^{-1/2}).
$$

Since
$$
(1,0,0)-\frac{1}{\mu}e(t)=O(|t|^{-1}),
$$
the minimal axis of the orbit-vertex
$T_{R^{m,n}_{1,\Omega_0,v_0}(t)}(1,0,0)$ is asymptotically not
greater than $O(t^{-1})$. Therefore, the length of the maximal
axis is asymptotically not greater than some function of type
$O(|t|^{-1/2})$. Hence, the orbit of the point $(1,0,0)$ is
contained in the $(C_1|t|^{-1/2})$-ball of the point $(1,0,0)$,
where $C_1$ is a constant that does not depend on $t$.

We have
$$
(0,1,0)-\frac{1}{t}g_{\max}(t)=O(|t|^{-1}).
$$
Therefore, the length of the maximal axis of the orbit-vertex
$T_{R^{m,n}_{1,\Omega_0,v_0}(t)}(1,0,0)$ is asymptotically not
greater than some function $1+O(t^{-1/2})$. Hence, the length of
the minimal axis is asymptotically not greater than some function
$O(|t|^{-1/2})$. This implies that the orbit of the point
$(0,1,0)$ is contained in the $(C_2|t|^{-1/2})$-tubular
neighborhood of the segment with vertices $(0,1,0)$ and
$(0,-1,0)$, where $C_2$ is a constant that does not depend on $t$.

Therefore, the convex hull of the union of two orbit-vertices
$$
T_{R^{m,n}_{1,\Omega_0,v_0}(t)}(1,0,0)\quad \hbox{and} \quad
T_{R^{m,n}_{1,\Omega_0,v_0}(t)}(0,1,0)
$$
is contained in the $C$-tubular neighborhood of the triangle with
vertices $(1,0,0)$, $(0,1,0)$, $(0,-1,0)$, where
$C=\max(C_1,C_2)|t|^{-1/2}$. This concludes the proof of the
lemma.
\end{proof}

Let us now formulate a similar statement for the general case of
Hessenberg matrices.

\begin{corollary}\label{general}
Let $\Omega=\langle a_{11},a_{21}|a_{12},a_{22},a_{32}\rangle$ and
$R^{m,n}_{1,\Omega,v}$ an NRS-ray. Then for any $\varepsilon>0$
there exists $C>0$ such that for any $t>C$ the convex hull of the
union of two orbit-vertices
$$
T_{R^{m,n}_{1,\Omega,v}(t)}(1,0,0)\quad \hbox{and} \quad
T_{R^{m,n}_{1,\Omega,v}(t)}(a_{11},a_{21},0)
$$
is contained in the $\varepsilon$-tubular neighborhood of the
convex hull of three points $(1,0,0)$, $(a_{11},a_{21},0)$,
$(-a_{11},-a_{21},0)$.
\end{corollary}

\begin{proof}
Denote $\Omega=\langle a_{11},a_{21}|a_{12},a_{22},a_{32}\rangle$
and choose
$$
X= \left(
\begin{array}{ccr}
a_{21}a_{32} & -a_{32}a_{11} & a_{11}a_{22}-a_{21}a_{12} \\
0              &a_{32}          & -a_{11}-a_{22}              \\
0              &0                & 1                             \\
\end{array}
\right).
$$

A direct calculation shows that
$$
H_{\langle
0,1|0,0,1\rangle}^{(1,0,0)}\left(l_1(m,n)-\frac{t}{a_{21}a_{32}},l_2(n_0)\right)
=XH_{\Omega}(m-t,n)X^{-1},
$$
where $l_1$ and $l_2$ are linear functions with coefficients
depending only on $a_{11}$, $a_{21}$, $a_{12}$, $a_{22}$, and
$a_{32}$. Therefore, the ray $R^{m,n}_{1,\Omega,v}$  after the
described change of coordinates and a homothety is taken to the
ray $R^{\tilde m, \tilde n}_{1,\Omega_0,v_0}$ of matrices with
Hessenberg type $\Omega_0=\langle 0,1|0,0,1\rangle$ for certain
$\tilde m$ and $\tilde n$.

Lemma~\ref{triangle} implies the following. For any
$\varepsilon>0$ there exists a positive constant such that for any
$t$ greater than this constant the convex hull of the union of two
orbit-vertices
$$
T_{R^{\tilde m, \tilde n}_{1,\Omega_0,v_0}(t)}(1,0,0)\quad
\hbox{and} \quad T_{R^{\tilde m, \tilde
n}_{1,\Omega_0,v_0}(t)}(0,1,0)
$$
is contained in the $\varepsilon$-tubular neighborhood of the
triangle with vertices $(1,0,0)$, $(0,1,0)$, $(0,-1,0)$.

Now if we reformulate the last statement for the family of
operators in old coordinates, then we get the statement of the
corollary.
\end{proof}

{\it Proof of Proposition~\ref{cfProp}}. We note that the operator
$R^{m,n}_{1,\Omega,v}(t)$ takes the point $(1,0,0)$ to the point
$(a_{11},a_{21},0)$. Therefore, the convex hull of the union of
two orbit-vertices
$$
T_{R^{m,n}_{1,\Omega,v}(t)}(1,0,0)\quad \hbox{and} \quad
T_{R^{m,n}_{1,\Omega,v}(t)}(a_{11},a_{21},0)
$$
(we denote it by $W(t)$) coincides with the set
$\Gamma_{R^{m,n}_{1,\Omega,v}(t)}^0(1,0,0)$.

From Proposition~\ref{orbitsCF} it follows that there exists a
fundamental domain for the continued fraction with all its
orbit-vertices contained in $W(t)$. Choose a sufficiently small
$\varepsilon_0$ such that the $\varepsilon_0$-tubular neighborhood
of the triangle with vertices
$$
(1,0,0), \qquad (a_{1,1},a_{2,1},0), \quad \hbox{and} \quad
(-a_{11},-a_{21},0)
$$
does not contain integer points distinct from the points of the
triangle. From Corollary~\ref{general} it follows that for a
sufficiently large $t$ the set $W(t)$ is contained in the
$\varepsilon_0$-tubular neighborhood of the triangle. This implies
the statement of Proposition~\ref{cfProp}.
\qed

\subsection{Geometry of Klein-Voronoi continued fractions for
matrices of $R^{m,n}_{2,\Omega,v}$}

Now let us study the remaining case of the rays of matrices with
asymptotic direction $(a_{11},a_{21})$.

\begin{proposition}\label{cfProp2}
Consider an NRS-ray $R^{m,n}_{2,\Omega,v}$. Then there exists
$C>0$ such that for any $t>C$ there exists a fundamental domain
for the Klein-Voronoi continued fraction of the matrix
$R^{m,n}_{2,\Omega,v}(t)$ such that all integer points in this
domain are contained in the triangle with vertices $(1,0,0)$,
$(-1,0,0)$, and $(a_{11},a_{21},0)$.
\end{proposition}

The proof of this proposition is based on the corollary of the
following lemma. We remind that $\Omega_0=\langle
0,1|0,0,1\rangle$.

\begin{lemma}\label{triangle2}
Let $R^{m,n}_{2,\Omega_0,v_0}$ be an NRS-ray. Then for any
$\varepsilon>0$ there exists $C>0$ such that for any $t>C$ the
convex hull of the union of two orbit-vertices
$$
T_{R^{m,n}_{2,\Omega_0,v_0}(t)}(1,0,0) \quad \hbox{and} \quad
T_{R^{m,n}_{2,\Omega_0,v_0}(t)}(0,1,0)
$$
is contained in the
$\varepsilon$-tubular neighborhood of the convex hull of three
points $(1,0,0)$, $(-1,0,0)$, $(0,1,0)$.
\end{lemma}

\begin{proof}
First, we note that the continued fractions for the operators $A$
and $A^{-1}$ coincide.

Secondly, the following holds:
$$
H_{\langle 0,1|0,0,1\rangle}^{(1,0,0)}(m,n+t)=
XH_{\langle 0,1|0,0,1\rangle}^{(1,0,0)}(-n-t,-m)X^{-1},
$$
where
$$
X= \left(
\begin{array}{ccr}
0&-1&-n-t \\
-1&0&-m   \\
0&0&-1    \\
\end{array}
\right).
$$

Thus, in the new coordinates we obtain the equivalent statement
for the ray $R^{-n,-m}_{1,\Omega_0,v_0}(t)$. Now
Lemma~\ref{triangle2} follows directly from Lemma~\ref{triangle}.
\end{proof}

\begin{corollary}\label{general2}
Let $\Omega=\langle a_{11},a_{21}|a_{12},a_{22},a_{32}\rangle$ and
$R^{m,n}_{2,\Omega,v}$ be an NRS-ray. Then for any $\varepsilon>0$
there exists $C>0$ such that for any $t>C$ the convex hull of the
union of two orbit-vertices
$$
T_{H_\Omega^{v}(m+a_{11}t,n+a_{21}t)}(1,0,0)\quad \hbox{and} \quad
T_{H_\Omega^{v}(m+a_{11}t,n+a_{21}t)}(a_{11},a_{21},0)
$$
is contained in the $\varepsilon$-tubular neighborhood of the
triangle with vertices $(1,0,0)$, $(-1,0,0)$, and
$(a_{11},a_{21},0)$. \qed
\end{corollary}

{\noindent {\it Remark.} We omit the proofs of
Corollary~\ref{general2} and Proposition~\ref{cfProp2}, since they
repeat the proofs of Corollary~\ref{general} and
Proposition~\ref{cfProp}.}

\subsection{Conclusion of the proof}

Let us finally conclude the proof of Theorem~\ref{maintheorem}.
Let $A$ be an operator with Hessenberg matrix $M$ in $SL(3,\z)$.
By Theorem~\ref{cfr} any $\varsigma$-reduced matrix congruent to
$M$ is constructed as the matrix $(M|v)$, where $v$ is an integer
vector in an arbitrary chosen fundamental domain of the
Klein-Voronoi continued fractions for $A$, in addition $v$ should
be the minimum of the absolute value of MD-characteristic on the
integer lattice except the origin. To calculate $\varsigma$-reduce
matrices we find all such minima of MD-characteristics in
appropriate fundamental domains.

\vspace{2mm}

{\noindent {\it The case of NRS-rays with asymptotic direction
$(-1,0)$}. Consider an NRS-ray $R^{m,n}_{1,\Omega,v}$. By
Proposition~\ref{cfProp} there exists $C>0$ such that for any
integer $t>C$ we can choose a fundamental domain for the
Klein-Voronoi continued fraction of $R^{m,n}_{1,\Omega,v}(t)$ such
that all its integer points are in the triangle with vertices
$(1,0,0)$, $(a_{11},a_{21},0)$, and $(-a_{11},-a_{21},0)$. }

This triangle contains only finitely many integer points, all of
them have the last coordinate equal to zero. The value of the
MD-characteristic for a point $(x,y,0)$ equals:
$$
(a_{21}x-a_{11}y)a_{32}^2y^2t+\tilde C,
$$
where the constant $\tilde C$ does not depend on $t$, it depends
only on $x$, $y$, and $\Omega$. Therefore, for any point $(x,y,0)$
the MD-characteristic is linear with respect to the parameter~$t$,
and it increases with growth of $t$. The only exceptions are the
points of type $\lambda(1,0,0)$ and $\mu(a_{11},a_{21},0)$ (for
integers $\lambda$ and $\mu$). The values of MD-characteristic are
constant in these points with respect to the parameter $t$.

Since there are finitely many integer points in the triangle
$(1,0,0)$, $(a_{11},a_{21},0)$, and $(-a_{11},-a_{21},0)$, for
sufficiently large $t$ the MD-characteristic at points of the
triangle attains the minima only at $(1,0,0)$ and at
$(a_{11},a_{21},0)$. Since $R^{m,n}_{1,\Omega,v}(t)$ takes the
point $(1,0,0)$ to the point $(a_{11},a_{21},0)$, a fundamental
domain may contain only one of these two points, let it be
$(1,0,0)$.

Therefore, for sufficiently large $t$ the minimum of
MD-characteristic at the integer points of the chosen fundamental
domain is unique and it is attained at point $(1,0,0)$. Hence by
Theorem~\ref{cfr} for sufficiently large $t$ the matrix
$$
\Big(H_{\langle
a_{11},a_{21}|a_{12},a_{22},a_{32}\rangle}^{(1,0,0)}(m-t,n)\Big|(1,0,0)\Big)=
H_{\langle
a_{11},a_{21}|a_{12},a_{22},a_{32}\rangle}^{(1,0,0)}(m-t,n)
$$
is the only $\varsigma$-reduced matrix in the conjugacy class.
This implies both statements of Theorem~\ref{maintheorem} for the
ray $R^{m,n}_{1,\Omega,v}$.

Therefore, Theorem~\ref{maintheorem} holds for any NRS-ray with
asymptotic direction $(-1,0)$.

\vspace{2mm}

{\noindent {\it The case of NRS-rays with asymptotic direction
$(a_{11},a_{21})$}. This case is similar to the case of NRS-rays
with asymptotic direction $(-1,0)$, so we omit the proof here.

}

Proof of Theorem~\ref{maintheorem} is completed. \qed

\section{Open problems}\label{Open}

In this section we formulate open questions on the structure of
the sets of NRS-matrices and briefly describe the situation for
RS-matrices.

\vspace{2mm}

{\noindent{\bf NRS-matrices.} As we have shown in
Theorem~\ref{maintheorem} the number of $\varsigma$-nonreduced
matrices in NRS-rays is always finite. Here we conjecture a
stronger statement.
\begin{conjecture}\label{Pr1}
Let $\Omega$ be an arbitrary Hessenberg type. All but a finite
number of NRS-matrices of type $\Omega$ are $\varsigma$-reduced.
\end{conjecture}
}

If the answer to this conjecture is positive we immediately have
the following general question.

\begin{problem}
Study the asymptotics of the number of $\varsigma$-nonreduced
NRS-matrices with respect to the growth of Hessenberg complexity.
\end{problem}

Denote the conjectured number of $\varsigma$-nonreduced
NRS-matrices of Hessenberg type $\Omega$ by $\#(\Omega)$. Numerous
calculations give rise to the following table for all types with
Hessenberg complexity less than $5$.

\vspace{2mm}

\begin{center}
\begin{tabular}{|c||c||c|c||c|c|c|}
\hline
$\Omega$&$\langle 0,1|0,0,1\rangle$ & $\langle 0,1|1,0,2\rangle$ &
$\langle 0,1|1,1,2\rangle$ & $\langle 0,1|1,0,3\rangle$ &
$\langle 0,1|1,1,3\rangle$ & $\langle 0,1|1,2,3\rangle$\\
\hline \hline
$\varsigma(\Omega)$ & $1$ & $2$ & $2$ & $3$ & $3$ & $3$ \\
\hline
$\#(\Omega)$        & $0$ & $12$ & $12$ & $6$ & $10$ & $10$ \\
\hline
\end{tabular}

\vspace{2mm}

\begin{tabular}{|c||c|c|c||c||c|c|}
\hline
$\Omega$ & $\langle 0,1|2,0,3\rangle$ & $\langle 0,1|2,1,3\rangle$
& $\langle 0,1|2,2,3\rangle$ &
$\langle 1,2|0,0,1\rangle$ & $\langle 0,1|1,0,4\rangle$ & $\langle 0,1|1,1,4\rangle$ \\
\hline \hline
$\varsigma(\Omega)$ & $3$ & $3$ & $3$ & $4$ & $4$ & $4$ \\
\hline
$\#(\Omega)$ & $14$ & $10$ & $10$ & $94$ & $6$ & $8$ \\
\hline
\end{tabular}

\vspace{2mm}

\begin{tabular}{|c||c|c|c|c|c|c||}
\hline
$\Omega$& $\langle 0,1|1,2,4\rangle$ & $\langle 0,1|1,3,4\rangle$
& $\langle 0,1|3,0,4\rangle$ & $\langle 0,1|3,1,4\rangle$ &
$\langle 0,1|3,2,4\rangle$ & $\langle 0,1|3,3,4\rangle$ \\
\hline \hline
$\varsigma(\Omega)$ & $4$ & $4$ & $4$ & $4$ & $4$ & $4$ \\
\hline
$\#(\Omega)$ & $10$ & $8$ & $10$ & $12$ & $8$ & $8$ \\
\hline
\end{tabular}

\end{center}

\vspace{2mm}

{\noindent{\bf RS-matrices.} We conclude this paper with a few
words about real spectra matrices (i.e., about $SL(3,\z)$-matrices
with three distinct real roots). Mostly we consider the family
$H(\langle 0,1|1,0,2\rangle)$, the situation with the other
Hessenberg types is similar.}

Recall that
$$
H_{\langle 0,1|1,0,2\rangle}^{(1,0,1)}(m,n)= \left(
\begin{array}{ccc}
0 &1 &n+1\\
1 &0 &m\\
0 &2 &2n+1\\
\end{array}
\right).
$$
This matrix is of Hessenberg type $\langle 0,1|1,0,2\rangle$, its
Hessenberg complexity equals 2. Hence $H_{\langle
0,1|1,0,2\rangle}^{(1,0,1)}(m,n)$ is $\varsigma$-reduced if and
only if it is not integer conjugate to some matrix of unit
Hessenberg complexity, all such matrices are of Hessenberg type
$\langle 0,1|0,0,1\rangle$.

In Figure~\ref{01_102.1} we show all matrices $H_{\langle
0,1|1,0,2\rangle}^{(1,0,1)}(m,n)$ with
$$
-20\le m,n \le 20.
$$
The square in the intersection of the $m$-th column with the
$n$-th row corresponds to the matrix $H_{\langle
0,1|1,0,2\rangle}^{(1,0,1)}(m,n)$. It is colored in black if the
characteristic polynomial has rational roots. The square is
colored in gray if the characteristic polynomial is irreducible
and there exists an integer vector $(x,y,z)$ with the coordinates
satisfying
$$
-1000\le x,y,z \le 1000,
$$
such that the MD-characteristic of $H_{\langle
0,1|1,0,2\rangle}{(1,0,1)}(m,n)$ equals $1$ at $(x,y,z)$. All the
rest squares are white.

If a square is gray, then the corresponding matrix is
$\varsigma$-nonreduced, see Remark~\ref{DavHes0}. If a square is
white, then we cannot conclude whether the matrix is
$\varsigma$-reduced or not (since the integer vector with unit
MD-characteristic may have coordinates with absolute values
greater than 1000).

\begin{figure}
$$\epsfbox{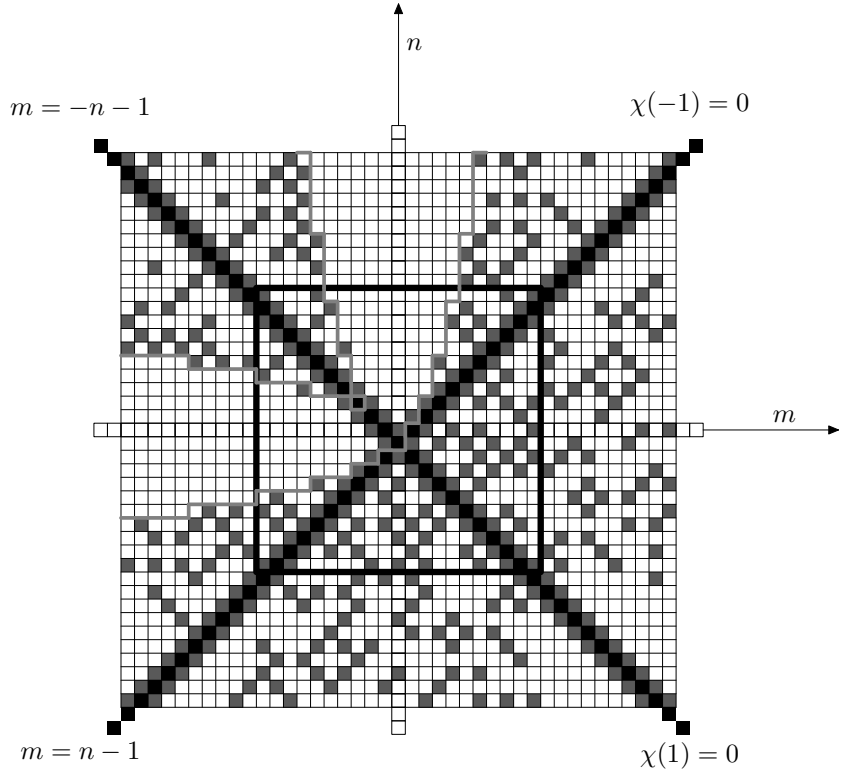}$$
\caption{The family of matrices of Hessenberg type $\langle
0,1|1,0,2\rangle$.}\label{01_102.1}
\end{figure}

It is most probable that white squares in Figure~\ref{01_102.1}
represent $\varsigma$-reduced matrices. We have checked explicitly
all the squares with
$$
-10\le m,n \le 10.
$$
These matrices are contained inside the big black square shown on
the figure. All white squares inside it correspond to
$\varsigma$-reduced matrices.

\vspace{1mm}

We show a boundary broken line between the NRS- and $RS$-squares
in gray.

\vspace{2mm}

{\noindent {\it Remark.} In Figure~\ref{01_102.1} the {\it
NRS-domain} is easily visualized, it almost completely consists of
white squares. While {\it RS-domain} contains relatively large
number of black squares. This indicates a significant difference
between RS- and NRS-cases.
}

\vspace{2mm}

Direct calculations of the corresponding MD-characteristic show
that the following proposition holds.

\begin{proposition}\label{odd}
If an integer $m{+}n$ is odd, then $H_{\langle
0,1|1,0,2\rangle}^{(1,0,1)}(m,n)$ is $\varsigma$-reduced. \qed
\end{proposition}

{\noindent {\it Remark.} From one hand Proposition~\ref{odd}
implies the existence of rays entirely consisting of
$\varsigma$-reduced matrices. From the other hand in contrast to
the NRS-matrices this is not always the case for RS-matrices. For
instance, all matrices corresponding to integer points of the
lines
$$
1)m=n; \quad 2)m=n+2; \quad 3)m=-n; \quad 4)m=-n-2; \quad 5)
n=3m-4; \quad 6)m=3n+6
$$
are $\varsigma$-reduced (we do note state that the list of such
lines is complete).}

\vspace{2mm}

So Theorem~\ref{maintheorem} does not have a direct generalization
to the RS-case and we end up with the following problem.
\begin{problem}
What is the percentage of $\varsigma$-reduced matrices among
matrices of a given Hessenberg type $\Omega$?
\end{problem}
It is more likely that almost all Hessenberg matrices are
$\varsigma$-reduced (except for some measure zero subset).

\bibliographystyle{plain}      
\bibliography{sl3zz}

\begin{thebibliography}{10}

\bibitem{Arnold1998}
V.~I. Arnold.
\newblock Higher-dimensional continued fractions.
\newblock {\em Regul. Chaotic Dyn.}, 3(3):10--17, 1998.
\newblock J. Moser at 70 (Russian).

\bibitem{Arnold1999}
V.~I. Arnold.
\newblock Preface.
\newblock In {\em Pseudoperiodic topology}, volume 197 of {\em Amer. Math. Soc.
  Transl. Ser. 2}, pages ix--xii. Amer. Math. Soc., Providence, RI, 1999.

\bibitem{Arnold2002}
V.~I. Arnold.
\newblock {\em Continued fractions}.
\newblock {M}oscow: {M}oscow {C}enter of {C}ontinuous {M}athematical
  {E}ducation, 2002.

\bibitem{Borevich1966}
A.~I. Borevich and I.~R. Shafarevich.
\newblock {\em Number theory}.
\newblock Translated from the Russian by Newcomb Greenleaf. Pure and Applied
  Mathematics, Vol. 20. Academic Press, New York, 1966.

\bibitem{Buchmann1985}
J.~Buchmann.
\newblock A generalization of {V}orono\u\i 's unit algorithm. {I}.
\newblock {\em J. Number Theory}, 20(2):177--191, 1985.

\bibitem{Buchmann1985a}
J.~Buchmann.
\newblock A generalization of {V}orono\u\i 's unit algorithm. {II}.
\newblock {\em J. Number Theory}, 20(2):192--209, 1985.

\bibitem{Davenport1938}
H.~Davenport.
\newblock On the product of three homogeneous linear forms. {I}.
\newblock {\em Proc. London Math. Soc.}, 13:139--145, 1938.

\bibitem{Davenport1938a}
H.~Davenport.
\newblock On the product of three homogeneous linear forms. {II}.
\newblock {\em Proc. London Math. Soc.(2)}, 44:412--431, 1938.

\bibitem{Davenport1939}
H.~Davenport.
\newblock On the product of three homogeneous linear forms. {III}.
\newblock {\em Proc. London Math. Soc.(2)}, 45:98--125, 1939.

\bibitem{Davenport1941}
H.~Davenport.
\newblock Note on the product of three homogeneous linear forms.
\newblock {\em J. London Math. Soc.}, 16:98--101, 1941.

\bibitem{Davenport1943}
H.~Davenport.
\newblock On the product of three homogeneous linear forms. {IV}.
\newblock {\em Proc. Cambridge Philos. Soc.}, 39:1--21, 1943.

\bibitem{Demmel1997}
J.~W. Demmel.
\newblock {\em Applied numerical linear algebra}.
\newblock Society for Industrial and Applied Mathematics (SIAM), Philadelphia,
  PA, 1997.

\bibitem{Hessenberg1942}
K.~Hessenberg.
\newblock {\em {T}hesis}.
\newblock {D}armstadt, {G}ermany: {T}echnische {H}ochschule, 1942.

\bibitem{Karpenkov2004}
O.~Karpenkov.
\newblock On the triangulations of tori associated with two-dimensional
  continued fractions of cubic irrationalities.
\newblock {\em Funct. Anal. Appl.}, 38(2):102--110, 2004.
\newblock {R}ussian version: Funkt. Anal. Prilozh. 38 (2), 2004, 28--37.

\bibitem{Karpenkov2004a}
O.~Karpenkov.
\newblock On two-dimensional continued fractions of hyperbolic integer matrices
  with small norm.
\newblock {\em Russian Math. Surveys}, 59(5):959--960, 2004.

\bibitem{Karpenkov2009b}
O.~Karpenkov.
\newblock Constructing multidimensional periodic continued fractions in the
  sense of {K}lein.
\newblock {\em Math. Comp.}, 78(267):1687--1711, 2009.

\bibitem{Karpenkov2010}
O.~Karpenkov.
\newblock On determination of periods of geometric continued fractions for
  two-dimensional algebraic hyperbolic operators.
\newblock {\em Math. Notes}, 88(1-2):28--38, 2010.
\newblock {R}ussian version: Mat. Zametki, 88(1), (2010), 30--42.

\bibitem{Karpenkov2012}
O.~Karpenkov.
\newblock {M}ultidimensional {G}auss {R}eduction {T}heory for conjugacy classes
  of {S}{L}$(n,${Z}$)$.
\newblock {\em Preprint}, 2012.

\bibitem{Klein1895}
F.~Klein.
\newblock {\"U}ber eine geometrische {A}uffassung der gew\"ohnliche
  {K}ettenbruchentwicklung.
\newblock {\em {N}achr. {G}es. {W}iss. {G}\"ottingen, {M}ath-phys. {K}l.},
  3:352--357, 1895.

\bibitem{Klein1896}
F.~Klein.
\newblock Sur une repr\'esentation g\'eom\'etrique de d\'eveloppement en
  fraction continue ordinaire.
\newblock {\em Nouv. Ann. Math.}, 15(3):327--331, 1896.

\bibitem{Kontsevich1999}
M.~L. Kontsevich and Yu.~M. Suhov.
\newblock Statistics of {K}lein polyhedra and multidimensional continued
  fractions.
\newblock In {\em Pseudoperiodic topology}, volume 197 of {\em Amer. Math. Soc.
  Transl. Ser. 2}, pages 9--27. Amer. Math. Soc., Providence, RI, 1999.

\bibitem{Korkina1995}
E.~I. Korkina.
\newblock Two-dimensional continued fractions. {T}he simplest examples.
\newblock {\em Trudy Mat. Inst. Steklov.}, 209(Osob. Gladkikh Otobrazh. s Dop.
  Strukt.):143--166, 1995.

\bibitem{Korkina1996}
E.~I. Korkina.
\newblock The simplest {$2$}-dimensional continued fraction.
\newblock {\em J. Math. Sci.}, 82(5):3680--3685, 1996.
\newblock Topology, 3.

\bibitem{Lachaud1993}
G.~Lachaud.
\newblock Poly\`edre d'{A}rnol'd et voile d'un c\^one simplicial: analogues du
  th\'eor\`eme de {L}agrange.
\newblock {\em C. R. Acad. Sci. Paris S\'er. I Math.}, 317(8):711--716, 1993.

\bibitem{Lachaud2002}
G.~Lachaud.
\newblock {\em Voiles et polyhedres de {K}lein}.
\newblock Act. Sci. Ind., Hermann, 2002.

\bibitem{Lewis1997}
J.~Lewis and D.~Zagier.
\newblock Period functions and the {S}elberg zeta function for the modular
  group.
\newblock In {\em The mathematical beauty of physics ({S}aclay, 1996)},
  volume~24 of {\em Adv. Ser. Math. Phys.}, pages 83--97. World Sci. Publ.,
  River Edge, NJ, 1997.

\bibitem{Manin2002}
Y.~I. Manin and M.~Marcolli.
\newblock Continued fractions, modular symbols, and noncommutative geometry.
\newblock {\em Selecta Math. (N.S.)}, 8(3):475--521, 2002.

\bibitem{Markoff1879}
A.~Markoff.
\newblock Sur les formes quadratiques binaires ind\'efinies.
\newblock {\em Math. Ann.}, 15(3-4):381--406, 1879.

\bibitem{Moussafir2000a}
J.-O. Moussafir.
\newblock {\em Voiles et {P}oly\'edres de {K}lein: {G}eometrie, {A}lgorithmes
  et {S}tatistiques}.
\newblock docteur en sciences th\'ese, Universit\'e Paris IX - Dauphine, 2000.

\bibitem{Ortega1963/1964}
J.~M. Ortega and H.~F. Kaiser.
\newblock The {$LL^{T}$} and {$QR$} methods for symmetric tridiagonal matrices.
\newblock {\em Comput. J.}, 6:99--101, 1963/1964.

\bibitem{Stoer2002}
J.~Stoer and R.~Bulirsch.
\newblock {\em Introduction to numerical analysis}, volume~12 of {\em Texts in
  Applied Mathematics}.
\newblock Springer-Verlag, New York, third edition, 2002.
\newblock Translated from the German by R. Bartels, W. Gautschi and C.
  Witzgall.

\bibitem{Trefethen1997}
L.~N. Trefethen and D.~Bau.
\newblock {\em Numerical linear algebra}.
\newblock Society for Industrial and Applied Mathematics (SIAM), Philadelphia,
  PA, 1997.

\bibitem{Voronoui1952}
G.~F. Vorono{\u\i}.
\newblock {\em On a {G}eneralization of the {A}lgorithm of {C}ontinued
  {F}raction. Collected works in three volumes ({I}n {R}ussian)}.
\newblock USSR Ac. Sci., Kiev., 1952.

\end{thebibliography}

\vspace{5mm}

\end{document}